\documentclass[12pt]{article}
\usepackage{amsmath,amssymb,amsthm}
\usepackage[margin=1in]{geometry}
\usepackage{hyperref}
\usepackage{tikz}
\usepackage{graphicx}
\usepackage{url}
\usepackage[mathlines]{lineno}
\usepackage{multirow}
\usepackage{dsfont} 
\usepackage{color}
\usepackage{multicol}
\definecolor{red}{rgb}{1,0,0}

\newtheorem{thm}{Theorem}[section]
\newtheorem{cor}[thm]{Corollary}
\newtheorem{lem}[thm]{Lemma}
\newtheorem{prop}[thm]{Proposition}

\newtheorem{obs}[thm]{Observation}

\theoremstyle{definition}
\newtheorem{defn}[thm]{Definition}

\theoremstyle{example}
\newtheorem{ex}[thm]{Example}

\DeclareMathOperator{\pr}{pr}
\DeclareMathOperator{\epr}{epr}

\DeclareMathOperator{\rank}{rank}
\DeclareMathOperator{\diag}{diag}

\newcommand{\R}{\mathbb{R}}
\newcommand{\F}{\mathbb{F}}

\newcommand{\Fnn}{\F^{n\times n}}

\newcommand{\Rnn}{\R^{n\times n}}

\newcommand{\bit}{\begin{itemize}}
\newcommand{\eit}{\end{itemize}}
\newcommand{\ben}{\begin{enumerate}}
\newcommand{\een}{\end{enumerate}}
\newcommand{\beq}{\begin{equation}}
\newcommand{\eeq}{\end{equation}}
\newcommand{\bea}{\begin{eqnarray*}}
\newcommand{\eea}{\end{eqnarray*}}
\newcommand{\bpf}{\begin{proof}}
\newcommand{\epf}{\end{proof}\ms}
\newcommand{\bmt}{\begin{bmatrix}}
\newcommand{\emt}{\end{bmatrix}}
\newcommand{\ms}{\medskip}

\newcommand{\ba}{\begin{array}}
\newcommand{\ea}{\end{array}}
\newcommand{\OL}{\overline}

\usepackage{enumitem}

\newcommand{\Ftnn}{\mathbb{F}_3^{n\times n}}

\begin{document}
\title{Combinatorial properties of the enhanced principal rank characteristic sequence over finite fields}

\author{
Peter J. Dukes\thanks{Department of Mathematics and Statistics,
University of Victoria,  Victoria, BC, V8W 2Y2, Canada (dukes@uvic.ca, martinez.rivera.xavier@gmail.com).}
\and
Xavier Mart\'inez-Rivera\footnotemark[1]
}


\maketitle

\begin{abstract}
The enhanced principal rank characteristic sequence (epr-sequence) of a symmetric matrix $B \in \mathbb{F}^{n \times n}$ is defined as $\ell_1 \ell_2 \cdots \ell_n$, where $\ell_j \in \{\tt{A}, \tt{S}, \tt{N}\}$ according to whether all, some but not all, 
or none of the principal minors of order $j$ of 
$B$ are nonzero. 
Building upon the second author's recent classification of the epr-sequences of symmetric matrices over the field $\mathbb{F}=\mathbb{F}_2$, we initiate a study of the case $\mathbb{F}=\mathbb{F}_3$.   
Moreover, epr-sequences over finite fields are shown to have connections to Ramsey theory and coding theory.
\end{abstract}


\medskip


\section{Introduction}\label{s: intro}
$\null$
\indent
For a given positive integer $n$, $[n]:=\{1,2, \dots,n\}$.
Let $\F$ be a field, $B \in \Fnn$ 
and $\alpha, \beta \subseteq [n]$; 
then $B$ has {\em order} $n$;
$B[\alpha, \beta]$ denotes the 
submatrix of $B$ lying in rows indexed by 
$\alpha$ and columns indexed by $\beta$;
$B[\alpha, \alpha]:= B[\alpha]$ and is called a
\textit{principal} submatrix;
the determinant of a 
$k \times k$ principal submatrix of $B$ is a
\textit{principal} minor, 
and such a minor has \textit{order} $k$.
With $\F_q$, where $q$ is a prime power, 
we denote the finite field with $q$ elements.

For a given symmetric matrix $B \in \Rnn$ and  
a fixed $k \in [n]$, Brualdi et al.\ \cite{P} introduced the principal rank characteristic sequence (pr-sequence) of $B$, in which they recorded the existence of at least one (respectively, the nonexistence  of any)  nonzero principal minor of order $k$ with a 1 (respectively, 0) in position $k$;
their definition of the pr-sequence was later extended to other fields by Barrett et al.\ \cite{BIRS13}:
Let $\F$ be a field and $B \in \Fnn$ be symmetric;
\textit{the principal rank characteristic sequence}
(\textit{pr-sequence}) of $B$ is defined as
$\pr(B) = r_0r_1 \cdots r_n$, where, 
for $k = 1,2,\dots,n$,
   \begin{equation*}
      r_k =
         \begin{cases}
             1 &\text{if $B$ has a nonzero principal minor of order $k$; and}\\
             0 &\text{otherwise;}
         \end{cases}
   \end{equation*}
while $r_0 = 1$ if and only if $B$ has a
$0$ diagonal entry.
As a simplification of the 
principal minor assignment problem (which is stated in \cite{HS}), 
Brualdi et al.\ \cite{P} studied what sequences of 0s and 1s are attained by symmetric matrices over the real field, $\R$, among other matrices.
The study of pr-sequences was extended by 
Barrett et al.\ \cite{BIRS13} to symmetric matrices over other fields, with their focus largely on fields of characteristic $2$.
Motivated by the pr-sequence,
and to cast more light on the existence or nonexistence of a nonzero principal minor of a given order in symmetric (and complex Hermitian) matrices,
Butler et al.\ \cite{EPR} introduced another sequence:
\begin{defn}
{\rm \cite[Definition 1.1]{EPR}}
{\rm
Let $\F$ be a field.
The {\em enhanced principal rank characteristic sequence} of $B \in \Fnn$   ({\em epr-sequence}) is defined as
$\epr(B)=\ell_1\ell_2 \cdots \ell_n$, where
\begin{equation*}
   \ell_j =
    \begin{cases}
             \tt{A} &\text{if \emph{all} of the principal minors of order $j$ are nonzero;}\\
             \tt{S} &\text{if \emph{some} but not all of the principal minors of order $j$ are nonzero;}\\
             \tt{N} &\text{if \emph{none} of the principal minors of order $j$ are nonzero (i.e., all are zero).}
         \end{cases}
\end{equation*}
}
\end{defn} 
Results concerning symmetric matrices over various fields, including constructions of matrices attaining certain epr-sequences, as well as results stating that certain subsequences cannot occur in the epr-sequence of a symmetric matrix, were presented in \cite{EPR}.
In \cite{XMR-Char 2}, Mart\'{i}nez-Rivera established the following (complete) characterization of the epr-sequences of symmetric matrices over the finite field $\F = \F_2$,  where, 
for a given sequence
$t_{i_1}t_{i_2} \cdots t_{i_k}$ from $\{\tt A,N,S\}$,
$\overline{t_{i_1}t_{i_2} \cdots t_{i_k}}$ indicates that
$t_{i_1}t_{i_2} \cdots t_{i_k}$ is repeated as many times as desired, or omitted entirely
(i.e., $\overline{t_{i_1}t_{i_2} \cdots t_{i_k}}$ may be vacuous):

\begin{thm}\label{F_2 theorem}
{\rm \cite[Theorems 3.2, 3.8, 3.11]{XMR-Char 2}}
Let $\sigma =\ell_1\ell_2 \cdots \ell_n$ be a
sequence from $\{\tt A,N,S\}$.
Then there exists a symmetric matrix $B \in \F_2^{n \times n}$ with
$\epr(B) = \sigma$ if and only if $\sigma$ is of
one of the following forms:
 \begin{multicols}{3}
    \begin{itemize}
\item[1.] $\tt A\OL{A}$\label{A1};
\item[2.] $\tt A\OL{S}N\OL{N}$\label{A2};
\item[3.] $\tt ASS\OL{S}A$\label{A3};
\item[4.] $\tt ASS\OL{S}AA$\label{A4};
\item[5.] $\tt ASSS\OL{S}AN$ \text{with $n$ even}\label{A5};
\item[6.] $\tt ASA\OL{SA}$\label{A6};
\item[7.] $\tt ASA\OL{SA}A$\label{A7};
\item[8.] $\tt ASA\OL{SA}N$\label{A8};
\item[9.] $\tt NA\OL{NA}$;
\item[10.] $\tt NA\OL{NA}N$;
\item[11.] $\tt \OL{NS}N\OL{N}$;
\item[12.] $\tt NS\OL{NS}NA$;
\item[13.] $\tt NS\OL{NS}NAN$;
\item[14.] $\tt S\OL{S}N\OL{N}$\label{S1};
\item[15.] $\tt S\OL{S}A$\label{S2};
\item[16.] $\tt S\OL{S}AA$\label{S3};
\item[17.] $\tt SS\OL{S}AN$\label{S4};
\item[18.] $\tt SASA\OL{SA}$\label{S5};
\item[19.] $\tt SASA\OL{SA}A$\label{S6};
\item[20.] $\tt SA\OL{SA}N$\label{S7}.
\item[\vspace{\fill}]
    \end{itemize}
    \end{multicols}
\end{thm}
Here, we build upon Theorem \ref{F_2 theorem}
by initiating a study of the epr-sequences of symmetric matrices over the field $\F=\F_3$,
as well as show that the epr-sequences of symmetric matrices over finite fields have connections to Ramsey theory and coding theory.
We note that a brief treatment of pr-sequences of symmetric matrices over $\F_3$ was presented in \cite[p.\ 231]{BIRS13}.

At the end of this section, 
an application of epr-sequences in coding theory,
one that is part of our motivation for studying the epr-sequences of symmetric matrices over finite fields, 
is identified.
Section \ref{s:prelim} contains preliminary results.
In Section \ref{s:Ramsey}, a connection between Ramsey theory and the epr-sequences of symmetric matrices over finite fields is established and used to obtain results for epr-sequences over finite fields.
Section \ref{s:F_3} focuses on the epr-sequences of symmetric matrices over the field $\F = \F_3$,
with its main result being a complete characterization of those that do not contain any $\tt S$ terms (see Theorem \ref{Main theorem}).
Before it concludes with Subsection \ref{s:Coding},
this section introduces some necessary terminology.

For a symmetric matrix $B$ with 
$\epr(B)=\ell_1\ell_2\cdots \ell_n$,
$[\epr(B)]_i := \ell_i$, for all $i \in [n]$;
moreover, for all $i,k \in [n]$  with $i+k \leq n$,
$[\epr(B)]_i^{i+k} := 
\ell_i \ell_{i+1}\ell_{i+2}\cdots \ell_{i+k}$. 
If a given sequence 
$\ell_1\ell_2 \cdots \ell_{n}$ from
$\{\tt A,N,S\}$ is the epr-sequence of a
symmetric matrix over a field $\F$, 
then we shall say that $\ell_1\ell_2 \cdots \ell_{n}$
is an ``epr-sequence over $\F$.''

Let $a$ and $b$ be positive integers.  The \textit{Ramsey number} $R(a,b)$ is the minimum integer $n$ such that, in any 2-coloring of the edges of the complete graph $K_n$, say with colors $\{\text{red},\text{blue}\}$, there is guaranteed to exist either a clique of order $a$, all of whose edges are red, or a clique of order $b$, all of whose edges are blue.  The case when $a=b$ is of particular importance; for instance, the well-known Ramsey number $R(3,3)=6$ asserts that any $\{\text{red},\text{blue}\}$ coloring of $K_6$ contains either a red or a blue triangle (while $K_5$ has a coloring that avoids this).  More generally, the `multicolor' Ramsey number $R_m(a)$ is defined as the minimum integer $n$ such that any $m$-colouring of the edges of $K_n$ induces a monochromatic clique of order $a$.   The interested reader is invited to read the excellent book \cite{Ramsey-book} for more information on Ramsey numbers, including the fact that they exist for any input parameters.  A dynamic survey of bounds on small Ramsey numbers is maintained at \cite{Ramsey-survey}.

For us, all graphs are simple, undirected and loopless.
The cycle graph on $n \geq 3$ vertices is 
denoted by $C_n$. 
Let $\F$ be a field and $G$ be a graph on $n$ vertices;
the number of vertices of $G$ is its \textit{order};
its complement is denoted by $\OL{G}$;
$G$ is {\em triangle-free} if 
it does contain $C_3$ as a subgraph;
and $A_{\F}(G)  \in \Fnn$ denotes 
the adjacency matrix of $G$ over the field $\F$.

For a given matrix $B \in \Fnn$ 
(where $\F$ is a field)  having a
nonsingular principal submatrix $B[\alpha]$, recall that
the Schur complement of $B[\alpha]$ in $B$ is the matrix
$B/B[\alpha] :=
B[\alpha^c] - B[\alpha^c,\alpha](B[\alpha])^{-1}B[\alpha,\alpha^c]$,
where $\alpha^c = [n] \setminus \alpha$ (see, for example, \cite{FZhang-Schur}).
Two given matrices $A$ and $B$ are said to be
{\em permutationally similar} if there exists a
permutation matrix $P$ such that $B=P^{\top}AP$;
moreover, the block diagonal matrix with 
the matrices $A$ and $B$ on 
the diagonal (in that order) is denoted by $A \oplus B$.
The $n\times n$ diagonal matrix whose 
$j$th diagonal entry is $d_j$ is 
denoted by $\diag(d_1,d_2, \dots,d_n)$.
With $I_n$, $O_{n}$ and $J_{m,n}$ we denote, respectively,
the $n \times n$ identity matrix,
the $n \times n$ zero matrix and
the $m \times n$ matrix all of whose entries are equal to $1$.
For integers $a$, $b$ and $p$, where $p >0$,
$a \equiv b \pmod{p}$ denotes
the congruence modulo $p$ of $a$ and $b$.

\subsection{An application in coding theory}\label{s:Coding}
$\null$
\indent
In this subsection,
an application in coding theory of epr-sequences of symmetric matrices over finite fields is identified.

A $q$-\textit{ary code} is a subset of $\F_q^n$ for some positive integer $n$, which is called the \textit{length} of the code.  A code $C \subseteq \F_q^n$ is \textit{linear} if it is a subspace of $\F_q^n$; in this case, we have $|C|=q^k$, where $k = \dim (C)$. The elements of $C$ are called \textit{codewords}.  A matrix $G \in \F_q^{k \times n}$ whose row space equals $C$ is called a \textit{generator matrix} for the linear code $C$.

The study of coding theory typically aims to control the number of positions in which two elements of $C$ differ.  With this in mind, we recall the \textit{Hamming distance} on $\F_q^n$ as $d(u,v) = |\{i : u_i \neq v_i\}|$. The \textit{minimum distance} of a code $C$ is $\min \{d(u,v): u,v \in C, u \neq v\}$.  For a linear code $C$, it is easy to see that its minimum distance also equals the least nonzero `weight' of a codeword; that is, $\min \{w(u):u \in C, u \neq 0\}$, where $w(u)=d(u,0)$ is the \textit{Hamming weight} of $u \in \F_q^n$.

Given a linear code $C \subseteq \F_q^n$, its \textit{dual code} is $C^\perp = \{v \in \F_q^n : u^\top v = 0 ~\forall\; u \in C\}$.  A generator matrix $H$ for $C^\perp$ is also called a \textit{parity-check matrix} for $C$.  Indeed, $H$ checks for membership in $C$ in the sense that $C = \text{ker}(H)$.  The connection with weights is as follows.  Given an $m \times n$ matrix with $m<n$, its \textit{spark} is the least integer $s$ such that there exists a set of $s$ linearly dependent columns in the matrix;  this integer $s$ can be viewed as the least weight of a nonzero vector in the kernel of the matrix.  It follows that the minimum distance of a linear code $C$ equals the spark of its parity-check matrix $H$.  

More generally, the \emph{weight enumerator} of a linear code $C$ of length $n$ is the polynomial
$$W(x,y) = \sum_{j=0}^n A_j x^j y^{n-j},$$
where $A_j=|\{u \in C: w(u) = j\}|$, the number of codewords of weight $j$.

The next result connects weights in a linear code to the epr-sequence of a certain {\em symmetric} matrix.

\begin{thm}\label{codes}
Let $C \subseteq \F_q^n$ be a nonzero linear code, $H$ be a parity-check matrix for $C$ and $\epr(H^\top H) = \ell_1\ell_2\cdots\ell_n$.  
For all $j \in [n]$, if $C$ has a codeword of weight $j$, then $\ell_j \neq \tt{A}$.  In particular, the minimum distance of $C$ is at least
$\min\{j \in [n]: \ell_j \neq {\tt A}\}$.
\end{thm}

\begin{proof}
Let $u \in C$ be a codeword of weight $j>0$  and $u=[u_1,u_2, \dots,u_n]^\top$, and put $\alpha=\{i:u_i \neq 0\}$ (such a $u$ exists because $C$ is nonzero). 
Observe that $|\alpha|=j$.
As $Hu=0$, the columns of $H$ that are indexed by $\alpha$ are linearly dependent. 
Let us denote the restriction of $H$ to columns in $\alpha$ by  $\widetilde{H}$ (i.e., $\widetilde{H}$ is the matrix obtained from $H$ by deleting the columns of $H$ that are not indexed by $\alpha$).
Then, as $(H^\top H)[\alpha] = \widetilde{H}^\top \widetilde{H}$ and $\text{rank}(\widetilde{H}^\top \widetilde{H}) \le \text{rank}(\widetilde{H}) < |\alpha|$, the $j \times j$ principal submatrix $(H^\top H)[\alpha]$ is singular, implying that $\ell_j \neq \tt A$.
\end{proof}

We remark that, over fields of characteristic zero, 
the spark of $H$ is exactly 
$\min\{j \in [n]: \ell_j \neq \tt A\}$,
where $\epr(H^\top H) = \ell_1\ell_2\cdots \ell_n$.  
In this sense, epr-sequences represent a refinement of spark.  Over finite fields, the converse of Theorem~\ref{codes} can fail, in the sense that $\ell_j \neq \tt{A}$ in spite of there being no codewords of weight $j$.  
In particular, it can be shown 
(using Theorems \ref{AA... over F2} and  \ref{AAA...})
that, for $C \subseteq \F_q^n$, where $q \in \{2,3\}$, 
we have $\min\{j \in [n]: \ell_j \neq {\tt A}\} \le q$, independent of the minimum distance of $C$.  
It may be of interest to explore how close this measure is to the minimum distance for larger $q$.

%

\section{Preliminary results}\label{s:prelim}
$\null$
\indent
This section contains preliminary results that are needed in subsequent sections, and it starts with some known facts.

\subsection{Results cited}
$\null$
\indent
The following well-known fact 
(see, for example, \cite{BIRS13})
states that the rank of a symmetric matrix $B$ is 
equal to the order of a 
largest nonsingular principal submatrix of $B$;
because of this, we shall call the rank of a 
symmetric matrix ``principal.''

\begin{thm}\label{Principal rank}
\label{thm: rank of a symm mtx}
{\rm \cite[Theorem 1.1]{BIRS13}}
Let $\F$ be a field and $B \in \Fnn$ be symmetric.
Then 
$\rank(B) = 
\max \{ |\gamma| : \det (B[\gamma]) \neq 0 \}$, where
the maximum over the empty set is defined to be 0.
\end{thm}

If $B$ is a given nonsingular symmetric matrix, 
then the epr-sequence of $B^{-1}$ is obtained readily from that of $B$:

\begin{thm}\label{Inverse Thm}
{\rm\cite[Theorem 2.4]{EPR} (Inverse Theorem.)}
Let $\F$ be a field and $B \in \Fnn$ be symmetric.
Suppose that $B$ is nonsingular.
If $\epr(B) = \ell_1\ell_2 \cdots \ell_{n-1}\tt{A}$, then
$\epr(B^{-1}) = \ell_{n-1}\ell_{n-2} \cdots \ell_{1}\tt{A}$.
\end{thm}

If $B$ is a given matrix and 
$C$ is a principal submatrix of $B$, 
then some of the terms in $\epr(C)$ may be 
deduced from those in $\epr(B)$, 
assuming that the latter is known:

\begin{thm}\label{Inheritance}
{\rm{\cite[Theorem 2.6]{EPR}}}
{\rm (Inheritance Theorem.)}
Let $\F$ be a field, $B \in \Fnn$ be symmetric, 
$m \leq n$ and $1\le i \le m$.
Then the following statements hold:
\ben
\item
If $[\epr(B)]_i={\tt N}$, then  $[\epr(C)]_i={\tt N}$,
for all $m\times m$ principal submatrices $C$.
\item If  $[\epr(B)]_i={\tt A}$, then  $[\epr(C)]_i={\tt A}$,
for all $m\times m$ principal submatrices $C$.

\item If $[\epr(B)]_m={\tt S}$, then there exist $m\times m$ principal submatrices $C_A$ and $C_N$ of $B$ such that $[\epr(C_A)]_m = {\tt A}$ and $[\epr(C_N)]_m = {\tt N}$.
\item If $i < m$ and $[\epr(B)]_i = {\tt S}$, then there exists an $m \times m$ principal submatrix $C_S$ such that $[\epr(C_S)]_i ={\tt S}$.
\een
\end{thm}

We shall need the following well-known fact about Schur complements.

\begin{thm}\label{Schur}
{\rm  (Schur Complement Theorem.)}
Let $\F$ be a field, $B \in \Fnn$ be symmetric,
$\alpha \subset [n]$, 
$\alpha^c = [n]\setminus \alpha$ and $|\alpha| = k$.
Suppose that $B[\alpha]$ is nonsingular,
and let $S = B/B[\alpha]$.
Then the following statements hold:
\begin{enumerate}
\item [$(i)$] 
$S$ is an $(n-k)\times (n-k)$ symmetric matrix.
\item [$(ii)$] 
{\rm \cite[p. 771]{Brualdi & Schneider}}
If the indexing of $S$ is inherited from $B$, then,
for all $\gamma \subseteq \alpha^c$,
\[
\det(S[\gamma]) = 
\frac{\det(B[\gamma \cup \alpha])}{\det(B[\alpha])}.
\]
\end{enumerate}
\end{thm}


If there are two consecutive $\tt N$s in 
the epr-sequence of a symmetric matrix, then each
letter in the sequence from that point forward is $\tt N$:

\begin{thm}\label{NN result}
{\rm \cite[Theorem 2.3]{EPR} ($\tt NN$ Theorem.)}
Let $\F$ be a field, $B \in \Fnn$ be symmetric and
$\epr(B) = \ell_1\ell_2 \cdots \ell_n$.
Suppose that $\ell_k = \ell_{k+1} = \tt{N}$, for some $k$.
Then, for all $j \geq k$, $\ell_j = \tt{N}$.
\end{thm}

The epr-sequence of any symmetric matrix over a field of characteristic not $2$ has neither $\tt NAN$ nor $\tt NAS$ as a subsequence:

\begin{thm}\label{NAN-NAS}
{\rm \cite[Theorem 2.14]{EPR}}
Let $\F$ be a field of characteristic not $2$, and 
let $B \in \Fnn$ be symmetric. 
Then neither {\tt NAN} nor {\tt NAS}
are subsequences of $\epr(B)$.
\end{thm}

We shall reference 
Theorem \ref{NAN-NAS} by saying that 
``{\tt NAN} is forbidden'' (likewise with {\tt NAS}).

The next theorem is concerned with symmetric matrices over the field $\F = \F_2$.

\begin{thm}\label{AA... over F2}
{\rm \cite[Theorem 2.9]{XMR-Char 2}}
Let $B  \in \F_2^{n \times n}$ be symmetric and 
$\epr(B)=\ell_1\ell_2 \cdots \ell_n$.
Suppose that $\tt AA$ is a subsequence of 
$\ell_1\ell_2 \cdots \ell_{n-1}$.
Then $\epr(B)=\OL{\tt A} \tt AAA\OL{\tt A}$.
\end{thm}

The following fact is well-known
(see, for example, \cite{Ramsey-book}).

\begin{thm}\label{triangle-free}
Let $G$ be a graph of order $5$.
If $G$ and $\OL{G}$ are triangle-free,
then $G = C_5$.
\end{thm}

\subsection{Other preliminary results}
$\null$
\indent

The fact that the rank of any matrix remains invariant after being multiplied by a nonzero constant, or after permuting one of its rows or columns, 
leads to an observation.

\begin{obs}\label{constant and permutation}
Let $\F$ be a field and $B \in \Fnn$ be symmetric.
Let $c \in \F$ be nonzero and 
$P \in \Fnn$ be a permutation matrix. 
Then $cB$ and $P^{\top}BP$ are symmetric,
$\epr(cB) = \epr(B)$ and $\epr(P^{\top}BP)=\epr(B)$.
\end{obs}

\begin{obs}\label{ones in first row}
Let $\F$ be a field and 
$B = [b_{ij}] \in \Fnn$ be symmetric.
Let $D=\diag(d_1,d_2,\dots,d_n) \in \Fnn$, 
$M=DBD$ and $M=[m_{ij}]$.
Then the following statements hold:
\ben
\item[$(i)$] $m_{ij}=d_id_jb_{ij}$, for all $i,j \in [n]$.
In particular, if $b_{1j} \neq 0$ for $j =2,3, \dots,n$ and
$D=\diag(1,b_{12}^{-1},b_{13}^{-1}, \dots,b_{1n}^{-1})$, then each off-diagonal entry in 
the first row (and first column) of $M$ is equal to $1$. 
\item[$(ii)$] $M$ is symmetric.
\item[$(iii)$] $\epr(M) = \epr(B)$.
\een
\end{obs}

Statement (i) in Observation \ref{ones in first row} is readily verified; 
statement (ii) follows from (i); and 
statement (iii) is a consequence of the fact that 
multiplying any row or column of a matrix by a nonzero constant preserves the rank of each of its submatrices.

The following fact generalizes 
\cite[Theorem 5.1]{EPR} and, to some extent,
\cite[Theorem 8.1]{P},
which provided the idea for our proof.

\begin{prop}\label{AN}
Let $\F$ be a field and $B \in \Fnn$ be symmetric.
Suppose that 
$\epr(B) = {\tt AN}\ell_3\ell_4\cdots\ell_n$.
Then there exists a symmetric matrix $M \in \Fnn$ such that the following statements hold:
\ben
\item $\epr(M) = \epr(B)$
\item Each entry of $M$ is equal to either $-1$ or $1$.
\item Each diagonal entry of $M$ and each entry in the first row and first column of $M$ is equal to $1$.
\een
\end{prop}

\bpf
It is readily verified that the assumption that 
$[\epr(B)]_1^2 = \tt AN$ implies that 
each entry of $B$ is nonzero.
Let $B=[b_{ij}]$, $C=b_{11}^{-1}B$ and $C=[c_{ij}]$.
It follows, then, that each entry of $C$ is nonzero;
in particular, observe that $c_{11} = 1$.
Let $D=\diag(c_{11}^{-1},c_{12}^{-1}, \dots,c_{1n}^{-1})$, $M=DCD$ and $M=[m_{ij}]$.
Observe that $M$ is symmetric.
By Observations \ref{constant and permutation} and \ref{ones in first row}, $\epr(M) = \epr(C) = \epr(B)$.

By Observation \ref{ones in first row}, 
$m_{11}=c_{11}^{-1}=1$ and 
each off-diagonal entry in 
the first row and first column of $M$ is $1$, as desired.
As $[\epr(M)]_2 = \tt N$, 
$0 =\det(M[\{1,j\}]) = m_{jj}-1$, for $j = 2,3, \dots,n$.
Thus, each diagonal entry of $M$ is $1$, as desired.

Then, as $[\epr(M)]_2 = \tt N$, 
$0 =\det(M[\{i,j\}]) = 1-m_{ij}^2$, 
for all $i,j \in [n]$ with $i \neq j$.
Hence, each off-diagonal entry of $M$ is 
either $-1$ or $1$.
\epf

Proposition \ref{AN} implies that to
find an $n \times n$ symmetric matrix 
(over a given field $\F$) 
whose epr-sequence is of the form 
${\tt AN}\ell_3\ell_4 \cdots \ell_n$
(if it exists), it suffices to search among the subset of symmetric matrices in $\Fnn$ all of whose entries 
are either $-1$ or $1$ and whose entries in 
the first row (and first column) and the diagonal are
all equal to $1$.

We draw upon the ideas of \cite[Theorem 2.2]{P} to make the following observation.

\begin{obs}\label{Jn-kIn}
Let $n \geq 1$ be an integer, 
$\F$ be a field and $k \in \F$ be nonzero.
Over the field $\F$, 
$\det(J_n - kI_n) = (-k)^{n-1}(n-k)$.
\end{obs}

We note that, 
as each $i \times i$ principal submatrix of $J_n-kI_n$
(for a given $i \in [n]$) is of the form $J_i-kI_i$,
all of the order-$i$ principal minors of 
$J_n-kI_n$ are equal,
implying that the epr-sequence of $J_n-kI_n$ does not have any $\tt S$ terms.
Thus, 
the next two facts follow from Observation \ref{Jn-kIn}
(the first of which generalizes \cite[Proposition 2.17]{EPR}).

\begin{prop}\label{Jn-In}
Let $\F$ be a field of characteristic $p > 0$, 
$J_n-I_n \in \Fnn$ and
$\epr(J_n-I_n) = \ell_1\ell_2 \cdots \ell_n$.
Then, for all $i \in [n]$, $\ell_i \in \{\tt A,N\}$.
Moreover, for all $i \in [n]$,
$\ell_i = \tt N$ if and only if $i \equiv 1 \pmod{p}$. 
\end{prop}

\begin{prop}\label{Jn-2In}
Let $\F$ be a field of characteristic $p > 2$, 
$J_n-2I_n \in \Fnn$ and
$\epr(J_n-2I_n) = \ell_1\ell_2 \cdots \ell_n$.
Then, for all $i \in [n]$, $\ell_i \in \{\tt A,N\}$.
Moreover, for all $i \in [n]$,
$\ell_i = \tt N$ if and only if $i \equiv 2 \pmod{p}$. 
\end{prop}

We draw upon the ideas of \cite[Proposition 2.5]{XMR-Classif} to establish the next fact.

\begin{prop}\label{ANA...}
Let $\F$ be a field of characteristic $p>2$ and
$B \in \Fnn$ be symmetric.
Suppose that
$\epr(B) = {\tt ANA}\ell_4 \ell_5 \cdots \ell_n$.
Then, for $i=4,5,\dots,n$, $\ell_i \in \{\tt A,N\}$.
Moreover, for $i=4,5,\dots,n$,
$\ell_i = \tt N$ if and only if $i \equiv 2 \pmod{p}$. 
\end{prop}

\bpf
Because of Proposition \ref{Jn-2In},
it suffices to show that $\epr(B) = \epr(J_n - 2I_n)$.
Because of Proposition \ref{AN}, we may assume,
without loss of generality, that each entry of $B$ is 
equal to either $-1$ or $1$, and that
each entry in the first row (and first column) and the diagonal of $B$ is equal to $1$.
We now show that 
\begin{equation}\label{ANA proof equation}
B=
\left(
\begin{array}{c|c}
1 & J_{1,n-1} \\
\hline
J_{n-1,1} & 2I_{n-1}-J_{n-1}
\end{array}
\right),
\end{equation}
by showing that each off-diagonal entry of the 
$(n-1)\times(n-1)$ submatrix $B[\{2,3,\dots,n\}]$ is
equal to $-1$.
Let $i,j \in \{2,3,\dots,n\}$ with $i \neq j$. 
Observe that $\det(B[\{1,i,j\}]) = -(b_{ij} - 1)^2$.
Then, as $[\epr(B)]_3 = \tt A$, 
and because $b_{ij} \in \{-1,1\}$, $b_{ij} = -1$.
Thus, (\ref{ANA proof equation}) holds.
Let $M$ be the (symmetric) matrix that results from multiplying both the first row and first column of $-B$ by $-1$.
As multiplying a row or a column of a matrix by a nonzero constant leaves the rank of each of its submatrices invariant, $\epr(M)=\epr(-B)$.
Observe that $M=J_n-2I_n$.
Then, as $\epr(-B) = \epr(B)$ 
(see Observation \ref{constant and permutation}), 
$\epr(B) = \epr(M) = \epr(J_n-2I_n)$. 
\epf

\section{Ramsey theory}\label{s:Ramsey}
$\null$
\indent
In this section, a connection between Ramsey theory and  epr-sequences of symmetric matrices over finite fields is established and used to obtain results for epr-sequences over finite fields.
The key idea is to interpret off-diagonal entries as edge colors.  In this way, Ramsey numbers give information on epr-sequences of sufficiently large matrices over finite fields.

We shall exploit the following observation in this section.

\begin{obs}\label{Ramsey idea}
Let $\F$ be a field,
$B \in \F^{n\times n}$ be symmetric and $B=[b_{ij}]$. 
Let $T \subseteq \F$ be finite and nonempty.
Suppose that, for all $i,j \in [n]$, 
$b_{ij} \in T$ if $i \neq j$.
If  $n \geq R_{|T|}(k)$,
then there exists $c \in T$ such that $B$ has a 
$k \times k$ principal submatrix all of whose 
off-diagonal entries are equal to $c$.
\end{obs}

We start by focusing on 
epr-sequences whose first term is $\tt N$.
Observe that, for a given matrix $B$ with 
$\epr(B) = \ell_1\ell_2\cdots\ell_n$ and
$\pr(B) = r_0r_1\cdots r_n$, 
$\ell_1 = \tt N$ if and only if $r_0r_1 = 10$;
and that, for $j=2,3,\dots,n$, 
$\ell_j = \tt N$ if and only if $r_j = 0$.
A characterization of the pr-sequences of symmetric matrices over fields of characteristic $2$ was established in \cite[Theorem 3.1]{BIRS13}, 
from which the following fact follows immediately:

\begin{thm}\label{epr with char 2}
Let $\F$ be a field of characteristic $2$ and
$B\in\Fnn$ be symmetric with 
$\epr(B)={\tt N}\ell_2 \ell_3\cdots \ell_n$.
If $j \in [n]$ is odd, then $\ell_j = \tt N$.
\end{thm}

With $p:=2$ (the characteristic of $\F$ in 
Theorem \ref{epr with char 2}),
the conclusion of 
Theorem \ref{epr with char 2} may be stated as follows: 
If $j \in [n]$ and $j \equiv 1 \pmod{p}$,
then $\ell_j = \tt N$;
something that is somewhat reminiscent of that can be said for finite fields:
Our next theorem leads to the following assertion:
If $B \in \F_q^{n \times n}$,
where $p$ is a prime and $q$ is a power of $p$, 
is a symmetric matrix whose epr-sequence starts with $\tt N$, then there exists an integer $m$ such that the following statement holds if $n \geq R_q(m)$: 
If $j \in [m]$ and $j \equiv 1 \pmod{p}$,
then $\ell_j \in \{\tt N,S\}$.

\begin{thm}\label{N... over Fq with Ramsey}
Let $p$ be a prime, $q$ be a power of $p$ and 
$B \in \F_q^{n \times n}$ be symmetric with
$\epr(B) = {\tt N}\ell_2\ell_3 \cdots \ell_n$.
Suppose that $n \geq R_{q}(kp+1)$,
for some positive integer $k$.
Then, for $i=1,2, \dots, k$, 
$\ell_{ip+1} \in \{\tt N,S\}$.
\end{thm}

\bpf
Let $i \in [k]$.
As $n \geq R_{q}(kp+1)$, 
and because each diagonal entry of $B$ is $0$,
Observation \ref{Ramsey idea} implies that
there exists $c \in \F_3$ such that 
$B$ has a $(kp+1) \times (kp+1)$ principal submatrix  all of whose off-diagonal entries are equal to $c$;
as each diagonal entry of $B$ is $0$, 
this submatrix is $c(J_{kp+1}-I_{kp+1})$.
Let $\alpha \subseteq [n]$ be such that
$B[\alpha] = c(J_{kp+1}-I_{kp+1})$.
Observe that each $(ip+1) \times (ip+1)$ principal submatrix of $B[\alpha]$ is of the form 
$c(J_{ip+1}-I_{ip+1})$,
whose determinant is $c^{ip+1}(-1)^{ip}(ip) = 0$
(see Observation \ref{Jn-kIn}).
Then, as each principal submatrix of $B[\alpha]$ is 
also a principal submatrix of $B$, 
$\ell_{ip+1} \neq \tt A$.
\epf

The following is a corollary to 
Theorem \ref{N... over Fq with Ramsey}, and 
it follows from the fact that
$R_3(4)=R(4,4,4) \leq 230$
(see, for example, \cite{Ramsey-survey}).

\begin{cor}
Let $B \in \F_3^{n \times n}$ be symmetric with
$\epr(B) = {\tt N}\ell_2\ell_3 \cdots \ell_n$.
Suppose that $n \geq 230$.
Then $\ell_4 \in \{\tt N,S\}$.
\end{cor}

We shall now consider the case in which the 
epr-sequence of the (symmetric) matrix $B$ in 
Theorem \ref{N... over Fq with Ramsey} starts with 
$\tt NA$; 
observe that, in that case, $B$ is not only 
a matrix all of whose diagonal entries are zero but also 
a matrix all of whose off-diagonal entries are nonzero.
We start by making an observation related to the third term in the epr-sequence of such a matrix.
If $B=[b_{ij}]$ is a symmetric matrix whose epr-sequence starts with $\tt NA$,
then $\det(B[\{p,q,r\}]) = 
2b_{pq}b_{pr}b_{qr}$ (if $p<q<r$).
The following observation, then, follows immediately.

\begin{obs}\label{NA... and ell_3}
Let  $B \in \F_q^{n \times n}$ be symmetric with
$\epr(B) = {\tt NA}\ell_3\ell_4 \cdots \ell_n$.
Then the following statements hold:
\ben
\item If $\F_q$ is of characteristic $2$, 
then $\ell_3=\tt N$.
\item If $\F_q$ is not of characteristic $2$, 
then $\ell_3=\tt A$.
\een
\end{obs}


The next theorem improves upon 
Theorem \ref{N... over Fq with Ramsey} in the case where $\epr(B)$ starts with $\tt NA$.

\begin{thm}\label{NA... over Fq with Ramsey}
Let $p$ be a prime, $q$ be a power of $p$ and 
$B \in \F_q^{n \times n}$ be symmetric with
$\epr(B) = {\tt NA}\ell_3\ell_4 \cdots \ell_n$.
Suppose that $n \geq R_{q-1}(k) + 1$,
for some positive integer $k$.
Then, for $i=1,2, \dots, k+1$, 
the following statements hold:
\ben
\item[$(i)$] If $i \equiv 1 \pmod{p}$, 
then $\ell_i \in \{\tt N,S\}$.
\item[$(ii)$] If $i \not\equiv 1 \pmod{p}$, 
then $\ell_i \in \{\tt A,S\}$.
\een
\end{thm}

\bpf
Let $i \in [k+1]$.
As $[\epr(B)]_1^2 = \tt NA$, 
each off-diagonal entry of $B$ is nonzero.
Because of Observation \ref{ones in first row},
we may assume, without loss of generality, that
each entry in the first row and first column of $B$ is equal to $1$.
As $n-1 \geq R_{q-1}(k)$,
Observation \ref{Ramsey idea} implies that
there exists a nonzero element $c \in \F_q$ such that
the $(n-1)\times(n-1)$ principal submatrix $B[\{2,3,\dots,n\}]$ has a $k \times k$ principal submatrix all of whose off-diagonal entries are equal to $c$;
as each diagonal entry of $B$ is $0$,
this $k \times k$ submatrix is $c(J_k-I_k)$.
Let $\alpha \subseteq \{2,3,\dots,n\}$ be such that
$B[\alpha] = c(J_k-I_k)$.
Let $M$ be the matrix that results from multiplying the first row and first column of $B$ by $c$.
It follows, then, that 
$M[\{1\} \cup \alpha] = c(J_{k+1}-I_{k+1})$.
Observe that each $i \times i$ principal submatrix of 
$M[\{1\} \cup \alpha] $ is of the form $c(J_{i}-I_{i})$,
whose determinant is
$\det(c(J_{i}-I_{i})) = 
c^{i}\det(J_{i}-I_{i}) =
c^{i}(-1)^{i-1}(i-1)$
(by Observation \ref{Jn-kIn}).
Then, as each principal submatrix of 
$M[\{1\} \cup \alpha]$ is also a 
principal submatrix of $M$,
it follows that
$\ell_i \neq \tt A$ if $i \equiv 1 \pmod{p}$, and that
$\ell_i \neq \tt N$ if $i \not\equiv 1 \pmod{p}$.
\epf

We shall rely on 
Observation \ref{NA... and ell_3} rather than
Theorem \ref{NA... over Fq with Ramsey} to make assertions about the third term in the epr-sequence of a symmetric matrix, as the former leads to stronger assertions.

The following is a corollary to 
Theorem \ref{NA... over Fq with Ramsey}, and 
it follows from the well-known fact that
$R_2(3)=R(3,3)=6$,
$R_2(4)=R(4,4)=18$ and
$R_2(5)=R(5,5) \leq 48$
(see, for example, \cite{Ramsey-survey}).

\begin{cor}\label{NA... over F3 with Ramsey}
Let  $B \in \F_3^{n \times n}$ be symmetric with
$\epr(B) = {\tt NA}\ell_3\ell_4 \cdots \ell_n$.
Then the following statements hold:
\ben
\item[$(i)$] If $n \geq 7$, then 
$\ell_4 \in \{\tt N,S\}$; and
\item[$(ii)$] if $n \geq 19$, then $\ell_5 \in \{\tt A,S\}$; and
\item[$(iii)$] if $n \geq 49$, then $\ell_6 \in \{\tt A,S\}$.
\een
\end{cor}

The following example shows that 
the first assertion in 
Corollary \ref{NA... over F3 with Ramsey} does not 
hold for $n<7$, and that
the second assertion does not hold for $n<19$.

\begin{ex}\label{NAAANA example}
\normalfont
For the following 
$6 \times 6$ symmetric matrix over $\F_3$,
it is readily verified that
$\epr(M_{\sigma}) = \sigma$:
\[
M_{\tt NAAANA}=
\left(
\begin{array}{c|ccccc}
 0 & 1 & 1 & 1 & 1 & 1 \\
 \hline
 1 & 0 & 1 & 2 & 2 & 1 \\
 1 & 1 & 0 & 1 & 2 & 2 \\
 1 & 2 & 1 & 0 & 1 & 2 \\
 1 & 2 & 2 & 1 & 0 & 1 \\
 1 & 1 & 2 & 2 & 1 & 0 \\
\end{array}
\right).
\]
\end{ex}

Observe that for fields of characteristic $2$, 
the assertions that one can make based on
Theorem \ref{NA... over Fq with Ramsey} are not as strong as those based on Theorem \ref{epr with char 2}.
Thus, we shall not state any corollaries to Theorem \ref{NA... over Fq with Ramsey} in the case where $q=4$ or any other power of $2$.

The next fact is another corollary to 
Theorem \ref{NA... over Fq with Ramsey}, and 
it follows from the well-known fact that
$R_4(3)=R(3,3,3,3) \leq 64$
(see, for example, \cite{Ramsey-survey}).

\begin{cor}
Let $B \in \F_5^{n \times n}$ be symmetric with
$\epr(B) = {\tt NA}\ell_3\ell_4 \cdots \ell_n$.
Suppose that $n \geq 65$.
Then $\ell_4 \in \{\tt A,S\}$.
\end{cor}

We close this section by applying Observation \ref{Ramsey idea} to make an assertion about  epr-sequences that start with $\tt AN$.

\begin{thm}\label{AN... over F3 with Ramsey}
Let $B \in \F_3^{n \times n}$ be symmetric with
$\epr(B) = {\tt AN}\ell_3\ell_4 \cdots \ell_n$.
Suppose that $n \geq R(3k+1,3k+1) + 1$,
for some positive integer $k$.
Then, for $i=1,2, \dots, k$, 
$\ell_{3i+2} \in \{\tt N,S\}$.
\end{thm}

\bpf
Let $i \in [k]$.
Because of Proposition \ref{AN}, 
we may assume that 
each entry of $B$ is equal to either $-1$ or $1$, and that 
each entry on the diagonal, 
the first row and the first column of $B$ is equal to $1$.
As $n-1 \geq R(3k+1,3k+1)$, 
Observation \ref{Ramsey idea} implies that
there exists a nonzero element $c \in \F_3$ such that 
$B[\{2,3,\dots n\}]$ (whose order is $n-1$) has a
$(3k+1) \times (3k+1)$ principal submatrix  all of whose off-diagonal entries are equal to $c$;
as each diagonal entry of $B$ is equal to $1$,
this submatrix is
$c(J_{3k+1}-I_{3k+1})+I_{3k+1}$.
Let $\alpha \subseteq \{2,3,\dots,n\}$ be such that 
$B[\alpha] = c(J_{3k+1}-I_{3k+1})+I_{3k+1}$.
As $c^2=1$ (because the square of 
any nonzero element in $\F_3$ is equal to $1$),
$B[\alpha] = c(J_{3k+1}-(1-c)I_{3k+1})$.
Let $M$ be the matrix that results from multiplying the first row and first column of $B$ by $c$.
As multiplying a row or a column of a matrix by a nonzero constant leaves the rank of each of its submatrices invariant, $\epr(M)=\epr(B)$. 
Thus, 
it suffices to establish the desired conclusion for $M$.
Observe that 
each entry in the first row and first column of $M$ is equal to $c$, and that 
the $(1,1)$-entry of $M$ is $c^2=1$.
Thus, 
$M[\{1\} \cup \alpha] = c(J_{3k+2}-(1-c)I_{3k+2})$.
Observe that 
each $(3i+2)\times (3i+2)$ principal submatrix of 
$M[\{1\} \cup \alpha]$ is of the form 
$c(J_{3i+2}-(1-c)I_{3i+2})$. 
By Observation \ref{Jn-kIn},
\begin{align*}
\det(c(J_{3i+2}-(1-c)I_{3i+2}))
&=c^{3i+2}\det(J_{3i+2}-(1-c)I_{3i+2})\\
&=c^{3i+2}(-(1-c))^{3i+1}(3i+2-(1-c)) \\
&=c^{3i+2}(c-1)^{3i+1}(3i+c+1).
\end{align*}
As $c$ is nonzero, either $c=1$ or $c=2$.
Thus, $\det(c(J_{3i+2}-(1-c)I_{3i+2})) = 0$.
Then, as each principal submatrix of $M[\{1\} \cup \alpha]$ is also a principal submatrix of $M$, 
$\ell_{3i+2} \neq \tt A$.
\epf

Applying Theorem \ref{AN... over F3 with Ramsey} and the fact that $R(4,4)=18$ (see, for example, \cite{Ramsey-survey}), leads to the following corollary.

\begin{cor}
Let $B \in \F_3^{n \times n}$ be symmetric with
$\epr(B) = {\tt AN}\ell_3\ell_4 \cdots \ell_n$.
Suppose that $n \geq 19$.
Then $\ell_5 \in \{\tt N,S\}$.
\end{cor}


\section{Symmetric matrices over $\mathbb{F}_3$}\label{s:F_3}
$\null$
\indent
This section focuses on the epr-sequences of symmetric matrices over the field $\F = \F_3$,
with its main result being a complete characterization of those that do not contain any $\tt S$ terms (see Theorem \ref{Main theorem}).

Before restricting our focus to the field $\F_3$,
we state two facts relating to 
not only $\F_3$ but also 
all fields of characteristic $3$,
the first of which follows immediately from 
Propositions \ref{Jn-In} and \ref{Jn-2In}:

\newpage

\begin{prop}\label{Jn-In and Jn-2In over F3}
The following statements hold over a field of characteristic $3$:
\ben
\item If $n \equiv 1 \pmod{3}$, then 
$\epr(J_n - I_n) = \tt \OL{NAA}N$ and
$\epr(J_n - 2I_n) = \tt \OL{ANA}A$.
\item If $n \equiv 2 \pmod{3}$, then 
$\epr(J_n - I_n) = \tt \OL{NAA}NA$ and
$\epr(J_n - 2I_n) = \tt \OL{ANA}AN$.
\item If $n \equiv 3 \pmod{3}$, then 
$\epr(J_n - I_n) = \tt \OL{NAA}NAA$ and
$\epr(J_n - 2I_n) = \tt \OL{ANA}ANA$.
\een
\end{prop}

The following statement follows immediately from Proposition \ref{ANA...}.

\begin{prop}\label{ANA... over F3}
Let $\F$ be a field of characteristic $3$ and
$B \in \Fnn$ be symmetric.
Suppose that
$\epr(B) = {\tt ANA}\ell_4 \ell_5 \cdots \ell_n$.
Then $\epr(B)$ is of one of the following forms:
\[
\tt ANA\OL{ANA}, \quad \quad
\tt ANA\OL{ANA}A, \quad \quad
\tt ANA\OL{ANA}AN.
\]
\end{prop}

The next fact will be used throughout these pages without being referenced.

\begin{obs}\label{square}
If $a \in \F_3$ is nonzero, then $a^2 =1$.
\end{obs}

An epr-sequence over $\F_3$ that 
starts with $\tt AA$ does not have, roughly speaking, 
any $\tt N$ terms in the first half of the sequence.

\begin{prop}\label{AA...}
Let $B  \in \Ftnn$ be symmetric.
Suppose that $\epr(B)={\tt AA}\ell_3\ell_4 \cdots \ell_n$.
Then, 
for all $3 \leq j \leq \big\lceil \frac{n}{2} \big\rceil$, 
$\ell_j \neq \tt N$.
\end{prop}

\bpf
As multiplying a matrix by a nonzero constant leaves its epr-sequence invariant 
(see Observation \ref{constant and permutation}), 
and because the square of 
any nonzero number in $\F_3$ is $1$,
we may assume, without loss of generality,
that the number of diagonal entries of $B$ that
are equal to $1$ is larger than the number of diagonal entries that
are equal to $2$.
Let $p$ be the number of diagonal entries of $B$ that are equal to $1$ (as $[\epr(B)]_1 = \tt A$, $p \neq 0$).
As any two permutationally similar matrices have
the same epr-sequence
(see Observation \ref{constant and permutation}),
we may assume, without loss of generality, that
all of the diagonal entries of the 
leading $p \times p$ principal submatrix of $B$ are 
equal to $1$.
As each diagonal entry of $B$ is nonzero,
the pigeonhole principle implies that
$p \geq \bigl\lceil \frac{n}{2} \bigr\rceil$.
Thus, it suffices to show that, 
for all integers $j$ with $3 \leq j \leq p$,
$\ell_j \neq \tt N$.
As $[\epr(B)]_2 = \tt A$,
$B[[p]] = I_p$ (otherwise, $B[[p]]$ would have zero principal minor of order $2$).
Thus,  for all integers $j$ with $3 \leq j \leq p$,
$B[[j]]= I_j$ (which is nonsingular), 
implying that $\ell_j \neq \tt N$.
\epf

A corollary to Proposition \ref{AA...} is that
a symmetric matrix over $\F_3$ whose epr-sequence starts with $\tt AAN$ is a matrix of order at most $4$:

\begin{cor}\label{AAN...}
Let $B  \in \Ftnn$ be symmetric.
Suppose that 
$\epr(B)={\tt AAN}\ell_4\ell_5 \cdots \ell_n$.
Then $n \leq 4$ and $\epr(B)$ is one of the following sequences:
\[
{\tt AAN}, \quad
{\tt AANA} \quad \text{or} \quad
{\tt AANN}.
\]
\end{cor}

\bpf
As $[\epr(B)]_3 = \tt N$,
Proposition \ref{AA...} implies that
$3>\big\lceil \frac{n}{2} \big\rceil$.
It follows, then, that $n \leq 4$, as desired.
Then, as $B$ is of order $n \geq 3$, 
and because $\ell_n \in \{\tt A,N\}$, either 
$\epr(B) =\tt AAN$ or 
$\epr(B) =\tt AANA$ or
$\epr(B) =\tt AANN$, as desired.
\epf

We need the following lemma.

\begin{lem}\label{AAA... lemma}
Let $B  \in \Ftnn$ be symmetric.
Suppose that $\epr(B)={\tt AAA}\ell_4\ell_5 \cdots \ell_n$ and 
$B \neq I_n$ and $B \neq 2I_n$.
Then there exists a constant $c \in \F$ and a 
permutation matrix $P$ such that
\[ 
cP^{\top}BP = 
\left(
\begin{array}{c|c}
I_p     & F \\
\hline
F^{\top} & 2I_{n-p}
\end{array}
\right),
\]
where
$\bigl\lceil \frac{n}{2} \bigr\rceil \leq p \leq n-1$  and 
$F$ is a matrix with at most one nonzero entry in 
each row and each column.
\end{lem}

\bpf
Suppose that $\epr(B)=\ell_1\ell_2 \cdots \ell_n$.
Thus, $\ell_1\ell_2\ell_3 = \tt AAA$.
As the square of any nonzero number in $\F_3$ is $1$,
there is a constant $c \in \F$ such that 
the number of diagonal entries of $cB$ that are equal to $1$ is 
larger than or equal to the number of diagonal entries that are equal to $2$.
Let $p$ be the number of diagonal entries of $cB$ that are equal to $1$.
As multiplying a matrix by a 
nonzero constant leaves its epr-sequence invariant
(see Observation \ref{constant and permutation}),
$\epr(cB) = \epr(B)$.
Then, as $\ell_1 = \tt A$,
the pigeonhole principle implies that
$p \geq \bigl\lceil \frac{n}{2} \bigr\rceil$.
Then, as $n \geq 3$, $p \geq 2$.
Let $P$ be a permutation matrix such that 
the first $p$ diagonal entries of
$P^{\top}cBP$ are equal to $1$.
Let $M = P^{\top}cBP$.
By Observation \ref{constant and permutation},
$\epr(M) = \epr(cB)=\epr(B)$.
If $p=n$, then the fact that
$\ell_2 = \tt A$ implies that $M = I_{n}$, 
which is a contradiction 
(because $B \neq I_{n}$ and $B \neq 2I_n$).
Thus,
$\bigl\lceil \frac{n}{2} \bigr\rceil \leq p \leq n-1$.
As $\ell_1 = \tt A$, 
each of the last $n-p$ diagonal entries of $M$ is 
equal to $2$.
Then, as $\ell_2 = \tt A$,
\[
M=
\left(
\begin{array}{c|c}
I_p     & F \\
\hline
F^{\top} & 2I_{n-p}
\end{array}
\right),
\]
for some matrix $F$.
Note that $F=M[[p], [n]\setminus [p]]$.
Let $M=[m_{ij}]$.
Observe that, for 
$r,s \in [p]$ with $r \neq s$ and 
$t \in [n] \setminus [p]$,
$\det(M[r,s,t]) = 2 - m_{rt}^2-m_{st}^2$;
then, as $\ell_{3} = \tt A$,
either $m_{rt}$ or $m_{st}$ is zero.
Thus, each column of $F$ contains 
at most one nonzero entry.
If $F$ has only one column, then the desired conclusion follows.
Thus, we assume that 
$F$ has at least two columns, implying that
$n-p \geq 2$.
Now observe that, for 
$r \in [p]$ and $s,t \in [n] \setminus [p]$ with $s\neq t$,
$\det(M[r,s,t]) = 2(2 - m_{rs}^2-m_{rt}^2)$;
then, as $\ell_{3} = \tt A$,
either $m_{rs}$ or $m_{rt}$ is zero.
Thus, $F$ contains 
at most one nonzero entry in each row.
\epf

With Lemma \ref{AAA... lemma} at our disposal, 
we are ready to demonstrate that, 
if a given epr-sequence over $\F_3$ starts with 
$\tt AAA$, then the sequence is $\tt AAA \OL{A}$.

\begin{thm}\label{AAA...}
Let $B  \in \Ftnn$ be symmetric.
Suppose that $\epr(B)={\tt AAA}\ell_4\ell_5 \cdots \ell_n$.
Then $\epr(B)=\tt AAA \OL{A}$.
\end{thm}

\bpf
As the desired conclusion follows immediately if 
$B=I_n$ or $B=2I_n$, assume that $B \neq I_n$ and $B \neq 2I_n$.
Because of Lemma \ref{AAA... lemma},
and because 
any two permutationally similar matrices have
the same epr-sequence,
and because the epr-sequence of a matrix is preserved after multiplication by a nonzero constant
(see Observation \ref{constant and permutation}),
we may assume, without loss of generality, that  
\[ 
B = 
\left(
\begin{array}{c|c}
I_p     & F \\
\hline
F^{\top} & 2I_{n-p}
\end{array}
\right),
\]
where 
$\bigl\lceil \frac{n}{2} \bigr\rceil \leq p \leq n-1$ and 
$F$ is a matrix with at most one nonzero entry in 
each row and each column.

As there is nothing to prove if $n=3$, 
we assume that $n \geq 4$.
Let $k$ be an integer with $4 \leq k \leq n$ and 
$C$ be a $k \times k$ principal submatrix of $B$.
It suffices to show that $C$ is nonsingular.
Observe that either 
$C=I_k$ or $C=2I_k$ or
\[ 
C = 
\left(
\begin{array}{c|c}
I_q     & G \\
\hline
G^{\top} & 2I_{k-q}
\end{array}
\right),
\]
where $1 \leq q \leq k-1$  and 
$G$ is a matrix with at most one nonzero entry in 
each row and each column.
Observe that
$G^{\top}G$ is a diagonal matrix all of 
whose diagonal entries are either $0$ or $1$.
As $C[[q]] = I_{q}$ is nonsingular,
the Schur Complement Theorem implies that
$\det(C) = \det(C[[q]]) \det(C/C[[q]])$,
where $C/C[[q]]$ is the 
Schur complement of $C[[q]]$ in $C$.
Thus, $\det(C) = \det(C/C[[q]])$,
implying that it suffices to show that 
$C/C[[q]]$ is nonsingular.
Observe that
\[
C/C[[q]] =
2I_{k-q} - G^{\top}(C[[q]])^{-1}G =
2I_{k-q} - G^{\top}(I_q)^{-1}G =
2I_{k-q} - G^{\top}G,
\]
which is a diagonal matrix all of whose diagonal entries are either $1$ or $2$,
implying that $C/C[[q]]$ is nonsingular, as desired.
\epf


We need another lemma.

\begin{lem}\label{NAXA...}
Let $B \in \Ftnn$ be symmetric.
Suppose that 
$\epr(B) = {\tt NA}\ell_3\ell_4\cdots \ell_n$.
If $\ell_4 = \tt A$, then 
$n \leq 6$ and $\epr(B)$ is one of the following sequences:
\[
{\tt NAAA}, \quad  \quad  
{\tt NAAAN}, \quad  \quad 
{\tt NAAANA}.
\]
\end{lem}

\bpf
Suppose that $\ell_4 = \tt A$ and 
$\epr(B) = \ell_1\ell_2 \cdots \ell_n$.
Thus, $\ell_1\ell_2 = \tt NA$.
By Corollary \ref{NA... over F3 with Ramsey}, 
$n \leq 6$.
As {\tt NAN} and {\tt NAS} are forbidden,
$\ell_3 = \tt A$.
If $n=4$, then $\epr(B) = \tt NAAA$, as desired.
Assume that $n \geq 5$.
If $\ell_5 \in \{\tt A,S\}$, then, 
by the Inheritance Theorem,
$B$ has a $5 \times 5$ principal submatrix $C$ with
$\epr(C) = \tt NAAAA$, which implies that 
$\epr(C^{-1}) = \tt AAANA$ (see the Inverse Theorem),
contradicting Theorem \ref{AAA...}.
It follows, then, that $\ell_5=\tt N$.
If $n=5$, then $\epr(B) = \tt NAAAN$, as desired.
Assume that $n = 6$.

We have already shown that 
$\epr(B) = {\tt NAAAN}\ell_6$.
Thus, it remains to show that $\ell_6 = \tt A$;
that is, it remains to show that $B$ is nonsingular.
As $\ell_1\ell_2 = \tt NA$, 
each off-diagonal entry of $B$ is nonzero.
Without loss of generality, 
we may assume that each off-diagonal entry in 
the first row and first column of $ B$ is equal to $1$
(see Observation \ref{ones in first row}).
Let $M$ be the $5 \times 5$ matrix obtained from 
$B[\{2,3,4,5,6\}]$ by replacing any 
off-diagonal entries that are equal to $2$ with $0$.
Let $G$ be the graph such that $A_{\F_3}(G) = M$
(as each diagonal entry of $B$ is zero, 
such a graph exists).
We will now show that both
$G$ and $\OL{G}$ are triangle-free;
it suffices to show that 
neither $J_3-I_3$ nor $2(J_3-I_3)$ is a 
principal submatrix of $B[\{2,3,4,5,6\}]$.
If $B[\{i,j,k\}] = J_3-I_3$ or $B[\{i,j,k\}] = 2(J_3-I_3)$
for some $i,j,k \in \{2,3,4,5,6\}$, then 
$B[\{1,i,j,k\}]$ is singular, 
which contradicts the fact that $\ell_4 = \tt A$.
Thus, both $G$ and $\OL{G}$ are triangle-free.
As $G$ is of order $5$, 
Theorem \ref{triangle-free} implies that $G=C_5$.
Thus, $B[\{2,3,4,5,6\}]$ is permutationally similar to 
$A_{\F_3}(C_5)  + 2A_{\F_3}(\OL{C_5})$.
Then, as each entry in 
the first row and first column of $ B$ is equal to $1$,
and because the first diagonal entry of $B$ is zero,
$B$ is uniquely determined, up to permutation similarity;
moreover, $B$ is permutationally similar to the nonsingular matrix in Example \ref{NAAANA example}.
It follows that $B$ is nonsingular, as desired.
\epf

A symmetric matrix over $\F_3$ whose epr-sequence contains $\tt NAAA$ as a subsequence is a matrix of order at most $6$:

\begin{prop}\label{NAAA}
If $\tt NAAA$ occurs as a subsequence of 
the epr-sequence of a 
symmetric matrix $B$ over $\F_3$, then 
$n \leq 6$ and 
$\epr(B)$ is one of the following sequences:
\[
{\tt NAAA}, \quad  \quad  
{\tt NAAAN}, \quad  \quad 
{\tt NAAANA}.
\]
\end{prop}

\bpf
Let $B$ be an $n \times n$ symmetric matrix over 
$\F_3$, and 
let $\epr(B) = \ell_1 \ell_2 \cdots \ell_n$.
Suppose that 
$\ell_k\ell_{k+1}\ell_{k+2}\ell_{k+3} = \tt NAAA$ 
for some $k \in [n-3]$.
Because of Lemma \ref{NAXA...}, 
it suffices to show that $k=1$.
Let $B'$ be a 
$(k+3)\times (k+3)$ principal submatrix of $B$, and 
let $\epr(B') = \ell'_1 \ell'_2 \cdots \ell'_{k+3}$.
By the Inheritance Theorem, 
$\epr(B') = \ell'_1\ell'_{2} \cdots \ell'_{k-1}\tt NAAA$.
By the Inverse Theorem, 
$\epr((B')^{-1}) = 
{\tt AAN}\ell'_{k-1}\cdots \ell'_{2}\ell'_{1}\tt A$.
It follows from Corollary \ref{AAN...} that 
$(B')^{-1}$ is of order at most $4$.
Thus, $k+3 \leq 4$, implying that $k=1$.
\epf

If the epr-sequence of a given symmetric matrix in $\F_3^{n \times n}$ starts with $\tt NAAN$, then
the sequence is completely determined by 
the value of $n$:

\begin{prop}\label{NAAN...}
Let $B \in \Ftnn$ be symmetric.
Suppose that 
$\epr(B) = {\tt NAAN}\ell_5\ell_6\cdots \ell_n$.
Then $\epr(B)$ is of one of the following forms:
\[
\tt NAAN\OL{AAN}, \quad \quad  
\tt NAAN\OL{AAN}A, \quad \quad
\tt NAAN\OL{AAN}AA.
\]
\end{prop}

\bpf
Because of Proposition \ref{Jn-In and Jn-2In over F3}, 
it suffices to show that $\epr(B) = \epr(J_n-I_n)$.
As $[\epr(B)]_1^2 = \tt NA$, 
each off-diagonal entry of $B$ is nonzero.
Because of Observation \ref{ones in first row},
we may assume that each off-diagonal entry in 
the first row and first column of $B$ is equal to $1$.
We claim that either
$B[\{2,3,\dots,n\}] = J_{n-1}-I_{n-1}$ or
$B[\{2,3,\dots,n\}] = 2(J_{n-1}-I_{n-1})$.
Suppose that this claim is false. 
As each off-diagonal entry of $B$ is nonzero, 
and because each diagonal entry of $B$ is zero,
and because $B$ is symmetric,
$B[\{2,3,\dots,n\}]$ has a row containing both a 
$1$ and a $2$ as off-diagonal entries.
Suppose that $B=[b_{ij}]$.
As any two permutationally similar matrices have 
the same epr-sequence
(see Observation \ref{constant and permutation}), 
we may assume, without loss of generality, that 
$b_{23} = 1$ and $b_{24} = 2$.
Observe that 
$\det(B[\{1,2,3,4\}]) = 1+(b_{34})^2$.
As $b_{34} \neq 0$, $(b_{34})^2 = 1$,
implying that $\det(B[\{1,2,3,4\}]) \neq 0$,
which contradicts the fact that $[\epr(B)]_4 = \tt N$. 
Thus, our claim is true.
We proceed by considering the 
two possible cases based on $B[\{2,3,\dots,n\}]$.

\noindent
\textbf{Case 1}: 
$B[\{2,3,\dots,n\}] = J_{n-1}-I_{n-1}$.

\noindent
As each off-diagonal entry in the 
first row and first column of $B$ is equal to $1$,
and because $b_{11} = 0$, $B=J_n-I_n$, 
implying that $\epr(B)=\epr(J_n-I_n)$, as desired.

\noindent
\textbf{Case 2}: 
$B[\{2,3,\dots,n\}] = 2(J_{n-1}-I_{n-1})$.

\noindent
Let $M$ be the matrix that results from 
multiplying the first row and first column of $B$ by $2$.
Then, as multiplication of any row or column
of a matrix by a nonzero constant leaves the rank of
every submatrix invariant, $\epr(B) = \epr(M)$.
As $M=2(J_n-I_n)$, 
and because $\epr(M)=\epr(J_n-I_n)$
(see Observation \ref{constant and permutation}),
$\epr(B)=\epr(J_n-I_n)$, as desired.
\epf

The next example exhibits two matrices of which we shall make use in the proof of Theorem \ref{Main theorem}.

\begin{ex}\label{Matrix examples for {A,N} theorem}
\normalfont
For the following symmetric matrices over $\F_3$, 
$\epr(M_{\sigma}) = \sigma$:
\[
M_{\tt AANA} = 
\left(
\begin{array}{cccc}
 1 & 0 & 1 & 1 \\
 0 & 1 & 1 & 1 \\
 1 & 1 & 2 & 0 \\
 1 & 1 & 0 & 2 \\
\end{array}
\right)
\quad \text{ and } \quad
M_{\tt AANN} = 
\left(
\begin{array}{cccc}
 1 & 0 & 1 & 1 \\
 0 & 1 & 1 & 2 \\
 1 & 1 & 2 & 0 \\
 1 & 2 & 0 & 2 \\
\end{array}
\right).
\]
\end{ex}

We are finally ready to establish the main result of this section.

\begin{thm}\label{Main theorem}
Let $n \geq 3$ and $\sigma = \ell_1\ell_2\cdots \ell_n$ be a sequence from $\{\tt A,N\}$.
Then there exists a symmetric matrix in $\Ftnn$ whose  epr-sequence is $\sigma$ 
if and only if 
$\sigma$ is of one of the following forms:
 \begin{multicols}{4}
    \begin{itemize}
\item[1.] $\tt AAA\OL{A}$.
\item[2.] $\tt AAN$.
\item[3.] $\tt AANA$.
\item[4.] $\tt AANN$.
\item[5.] $\tt ANA\OL{ANA}$.
\item[6.] $\tt ANA\OL{ANA}A$.
\item[7.] $\tt ANA\OL{ANA}AN$.
\item[8.] $\tt ANN\OL{N}$.
\item[9.] $\tt NAA$.
\item[10.] $\tt NAAA$.
\item[11.] $\tt NAAAN$.
\item[12.] $\tt NAAANA$.
\item[13.] $\tt NAAN\OL{AAN}$. 
\item[14.] $\tt NAAN\OL{AAN}A$. 
\item[15.] $\tt NAAN\OL{AAN}AA$.
\item[16.] $\tt NNN\OL{N}$.
 \end{itemize}
    \end{multicols}
\end{thm}

\bpf
Suppose that there exists a 
symmetric matrix $B \in \Ftnn$ with $\epr(B) = \sigma$.
By the {\tt NN} Theorem, 
$\ell_1\ell_2\ell_3 \neq \tt NNA$.
Then, as $\tt NAN$ is forbidden,
$\ell_1\ell_2\ell_3 \in 
\{\tt AAA,AAN, ANA, ANN, NAA, NNN\}$.
If $\ell_1\ell_2\ell_3 = \tt AAA$, 
then $\sigma = \tt AAA\OL{A}$ 
(see Theorem \ref{AAA...}).
If $\ell_1\ell_2\ell_3 = \tt AAN$,
Corollary \ref{AAN...} implies that either
$\sigma = \tt AAN$ or
$\sigma = \tt AANA$ or 
$\sigma = \tt AANN$.
If $\ell_1\ell_2\ell_3 = \tt ANA$, 
then, by Proposition \ref{ANA... over F3}, either
$\sigma = \tt ANA\OL{ANA}$ or
$\sigma = \tt ANA\OL{ANA}A$ or
$\sigma = \tt ANA\OL{ANA}AN$.
If $\ell_1\ell_2\ell_3 = \tt ANN$ or
$\ell_1\ell_2\ell_3 = \tt NNN$, 
then the {\tt NN} Theorem implies that either 
$\sigma = \tt ANN\OL{N}$ or
$\sigma = \tt NNN\OL{N}$.

Finally, we consider the remaining case:
$\ell_1\ell_2\ell_3 = \tt NAA$.
If $n=3$, then $\sigma = \tt NAA$. 
Suppose that $n \geq 4$.
We proceed by considering two cases, based on $\ell_4$.

\noindent
\textbf{Case 1}: $\ell_4 = \tt A$.

\noindent
It follows that
$\sigma = {\tt NAAA}\ell_5\ell_6\cdots \ell_n$ 
and, therefore, by Proposition \ref{NAAA}, either
$\sigma = \tt NAAA$ or
$\sigma = \tt NAAAN$ or
$\sigma = \tt NAAANA$.

\noindent
\textbf{Case 2}: $\ell_4 = \tt N$.

\noindent
It follows from Proposition \ref{NAAN...} that 
either
$\sigma = \tt NAAN\OL{AAN}$ or
$\sigma = \tt NAAN\OL{AAN}A$ or
$\sigma = \tt NAAN\OL{AAN}AA$.

We will now establish the other direction.
If $\sigma$ is of the form $\tt AAA\OL{A}$, 
then $\sigma = \epr(I_n)$.
If $\sigma = \tt AAN$, then, 
by the Inheritance Theorem, 
$\sigma$ is the epr-sequence of any of the $3 \times 3$ principal submatrices of the matrix $M_{\tt AANA}$ in 
Example \ref{Matrix examples for {A,N} theorem}.
If $\sigma = \tt AANA$ or $\sigma = \tt AANN$,
then $\sigma$ is the epr-sequence of one of the matrices in Example \ref{Matrix examples for {A,N} theorem}.
If $\sigma$ is of one of the following forms, then
$\sigma = \epr(J_n-2I_n)$:
$\tt ANA\OL{ANA}$,
$\tt ANA\OL{ANA}A$, and
$\tt ANA\OL{ANA}AN$
(see Proposition \ref{Jn-In and Jn-2In over F3}).
If $\sigma$ is of the form $\tt ANN\OL{N}$, 
then $\sigma = \epr(J_n)$.
If $\sigma=\tt NAA$, then $\sigma = \epr(J_3-I_3)$.
If $\sigma=\tt NAAA$, 
then $\sigma = \epr((M_{\tt AANA})^{-1})$
(see the Inverse Theorem).
If $\sigma=\tt NAAAN$, 
then, by the Inheritance Theorem, 
$\sigma$ is the epr-sequence of any of 
the $5 \times 5$ principal submatrices of 
the matrix $M_{\tt NAAANA}$ in 
Example \ref{NAAANA example};
and, if $\sigma=\tt NAAANA$, 
then $\sigma = \epr(M_{\tt NAAANA})$.
If $\sigma$ is of one of the following forms, then
$\sigma = \epr(J_n-I_n)$:
$\tt NAAN\OL{AAN}$,
$\tt NAAN\OL{AAN}A$ and 
$\tt NAAN\OL{AAN}AA$
(see Proposition \ref{Jn-In and Jn-2In over F3}).
If $\sigma$ is of the form $\tt NNN\OL{N}$, 
then $\sigma = \epr(O_n)$.
\epf

\subsection*{Acknowledgments}
$\null$
\indent
The research of Peter Dukes is supported by the 
NSERC grant RGPIN-312595--2017.
The research of Xavier Mart\'{i}nez-Rivera was 
supported by the following NSERC grants:
RGPIN-3677-2016 and RGPIN-312595-2017.


\end{document}